\documentclass[letterpaper, 10pt, conference]{ieeeconf}

\usepackage{amsmath,amssymb,bbm,bm,amsthm}
\usepackage{mathrsfs}
\usepackage[linesnumbered,ruled,noend]{algorithm2e}
\newtheorem{problem}{Problem}

\theoremstyle{remark}
\newtheorem{Remark}{Remark}

\newtheorem{Lemma}{Lemma}

\newtheorem{Theorem}{Theorem}

\usepackage{graphicx}
\usepackage[colorinlistoftodos]{todonotes}
\usepackage[colorlinks=true, allcolors=blue]{hyperref}
\renewcommand{\Re}{\mathbb{R}}

\newcommand{\Expectation}{\mathbb{E}}
\newcommand{\trace}{\mathtt{tr}}

\newcommand{\ScriptA}{\mathcal{A}}
\newcommand{\ScriptB}{\mathcal{B}}
\newcommand{\ScriptD}{\mathcal{D}}

\usepackage[caption=false]{subfig}

\IEEEoverridecommandlockouts
\title{\LARGE \bf Optimal Stochastic Vehicle Path Planning Using Covariance Steering}

\author{Kazuhide Okamoto\thanks{K. Okamoto is with the School of Aerospace Engineering, Georgia Institute of Technology, Atlanta, GA 30332-0150, USA. Email: kazuhide@gatech.edu} \qquad
	Panagiotis Tsiotras\thanks{P. Tsiotras is with the School of Aerospace Engineering, and also with the Institute for Robotics and Intelligent Machines, Georgia Institute of Technology, Atlanta, GA 30332-0150, USA. Email: tsiotras@gatech.edu}}
\begin{document}
	\maketitle
	\thispagestyle{empty}
	\pagestyle{empty}
	
\begin{abstract}
	This work addresses the problem of vehicle path planning in the presence of obstacles and uncertainties, which is a fundamental problem in robotics. 
	While many path planning algorithms have been proposed for decades, many of them have dealt with only deterministic environments or only \emph{open-loop} uncertainty, i.e., the uncertainty of the system state is not controlled and, typically, increases with time due to exogenous disturbances, which leads to the design of potentially conservative nominal paths. 
	In order to deal with disturbances and reduce uncertainty, generally, a lower-level feedback controller is used.
	We conjecture that, if a path planner can consider the \emph{closed-loop} evolution of the system uncertainty, it can compute less conservative but still feasible paths.
	To this end, in this work we develop a new approach that is based on optimal covariance steering, which explicitly steers the state covariance for stochastic linear systems with additive noise under non-convex state chance constraints.
	The proposed framework is verified using simple numerical simulations.
\end{abstract}

\section{Introduction}\label{sec:Introduction}
Vehicle path planning problems have been an active research topic for more than a decade including spacecraft~\cite{richards2002spacecraft}, unmanned aerial vehicles (UAVs)~\cite{hao2005differential}, and self-driving vehicles~\cite{kuwata2009real}.
A good comprehensive reference is \cite{LaValle:06}.
Among the existing literature, sampling-based algorithms such as the rapidly-exploring random trees (RRT)~\cite{lavalle1998rapidly}  have been popular for solving motion planning problems.
As the original RRT algorithm does not have any guarantees that the solution converges to the global optimum, variants of RRT have been proposed that offer asymptotic optimality guarantees, such as RRT$^\ast$~\cite{karaman2011sampling} and RRT$^\sharp$~\cite{arslan2013use}.
However, in general, these approaches deal with deterministic dynamics and cannot deal with uncertain systems. 

In order to solve path planning problems under uncertainty, several approaches have been proposed, such as the chance-constrained RRTs~\cite{luders2010chance,luders2013robust}, and mixed-integer programming approaches~\cite{blackmore2011chance,pinto2017risk}.
These approaches consider only the \emph{open-loop} dynamics of the covariance, i.e., the covariance of the system state is not controlled and, typically, increases with time due to the disturbance. 
In order to add robustness with respect to disturbances, state-feedback \emph{closed-loop} controllers are applied \emph{after the fact}.
Thus, it is more natural to consider the closed-loop evolution of the covariance for path planning.

Recently, path-planning problems with closed-loop covariance dynamics have been addressed by several researchers, such as~\cite{van2011lqg,bry2011rapidly}.
Vitus and Tomlin~\cite{vitus2011closed} also addressed a similar problem using mixed-integer programming (MIP).
Our work has connections with this line of work.

MIP approaches have also been actively investigated for path-planning problems in non-convex state constraint environments.
For example, mixed-integer linear programming (MILP) has been used for path planning problems~\cite{richards2002spacecraft,schouwenaars2001mixed}.
In addition, mixed-integer quadratic programming (MIQP) has been employed for path planning of multiple UAVs~\cite{mellinger2012mixed}, and mixed-integer semi-definite programming (MISDP) was used for UAV path planning~\cite{deits2015efficient}.
Note that these MIP-based approaches do not explicitly consider system uncertainty. 
The MIP approaches that cope with system uncertainty have been discussed by several researchers, such as~\cite{blackmore2011chance,pinto2017risk}.
However, these works consider only the \emph{open-loop} dynamics of the covariance, which may lead to unnecessarily conservative solutions.
Vitus et al.~\cite{vitus2008locally} dealt with a path-planning problem with \emph{closed-loop} dynamics of the covariance using the so-called Tunnel-MILP approach, which decomposes a non-convex environment into convex polygons and solves the optimal control problem through the convex polygons.
In this formulation, integer variables indicate in which polygon the state variable belongs to at each time step.
This is computationally more efficient than other MIP approaches, which typically need separate integer variables for every face of every obstacle.
While the original Tunnel-MILP approach could not cope with constraint violations between time steps, Deits and Tedrake~\cite{deits2014footstep} proposed a new constraint such that two consecutive system states have to belong to the same convex polygon in order for the system state to remain in the same polygon between consecutive time steps, implying no collision with obstacles.

These previous MIP approaches do not consider the terminal state distribution.
In fact, steering the covariance to a pre-specified value at a given time horizon needs to be formulated as a numerical optimization problem \cite{okamoto2018Optimal}.
The approach proposed in this work deals with the closed-loop dynamics of the covariance, computes the collision-free path under non-convex state chance constraints, and steers the system state to a pre-specified Gaussian distribution while minimizing a state and control expectation-dependent cost.

\subsection{Optimal Covariance Control}
In this work we design an optimal path planner that steers the mean and the covariance of a stochastic time-varying discrete linear system from initial values to pre-specified terminal values at a given time step in the presence of obstacles and uncertainties.
We assume that an initial-state Gaussian distribution and a state and input-independent white-noise Gaussian diffusion with known statistics are given a priori.
The system is controlled with the aim of steering the state to the target Gaussian distribution, while minimizing a state and control expectation-dependent cost.
In addition, to deal with obstacles in the environment, we consider the task under non-convex state constraints.
As the diffusion is unbounded, we need to probabilistically formulate the state constraints in terms of chance constraints.

Controlling the state covariance of a linear system has been researched since the late `80s. 
Hotz and Skelton~\cite{hotz1985covariance,hotz1987covariance} were the first authors to introduce the so-called covariance steering (or covariance assignment) problem and computed the state feedback gains that steer the steady-state covariance of a linear time-invariant system to converge to a pre-specified value.
Since then, many works have been devoted to the \emph{infinite} horizon covariance assignment problem, both for continuous and discrete time systems~\cite{iwasaki1992quadratic,xu1992improved}. 
It is interesting to note that, until recently, the \emph{finite}-horizon covariance control problem, in which the controller steers the system covariance to a pre-specified value at a given time step, had not been investigated.
Chen et al.~\cite{chen2016optimalI,chen2016optimalII,chen2018optimalIII} related the optimal covariance steering problem with the problems of Schr\"{o}dinger bridges~\cite{Schrodinger1931} and the optimal mass transfer~\cite{Kantorovich1942}.
Ridderhof and Tsiotras used a similar approach for the stochastic control of a Mars lander during powered descent~\cite{ridderhof2018uncertainty}.
Other approaches regard the finite horizon covariance controller as a linear quadratic Gaussian (LQG) controller with particular weights~\cite{goldshtein2017finite}. 
This approach can also be formulated and solved as a convex programming problem~\cite{bakolas2016optimalCDC,bakolas2018finite}.

Our previous work~\cite{okamoto2018Optimal} converted the finite-horizon optimal control problem into a convex programming problem in the presence of state chance constraints.
Chance constraints are used for stochastic control problems and impose maximum probability of constraint violation instead of hard constraints.
This formulation is useful for systems with potentially unbounded disturbance, because for such systems hard constraints are meaningless.
Chance-constrained optimization has been extensively studied since the `50s, with the aim of designing systems with guaranteed performance under uncertainty~\cite{Geletu2013}.
Various techniques have been proposed to solve the stochastic model-predictive control (for example, see~\cite{Farina2016} for an extensive review).
In this work, in order to cope with state chance-constraints, we follow the approach proposed in~\cite{blackmore2009convex}, where, using Boole's inequality~\cite{prekopa1988boole}, the authors showed that, in the case of Gaussian additive disturbance and a linear system, the optimal control problem can be converted into a convex programming problem with little conservatism.

The contributions of this work are twofold.
The first contribution is to solve the optimal covariance steering problem under \emph{non-convex} state chance constraints, which, to the best of our knowledge, has not been addressed in the literature.
The proposed approach facilitates the path planner that computes the optimal trajectory in the presence of obstacles and uncertainties.
We use mixed-integer convex programming and efficiently decomposes the admissible state space to a union of overlapping convex sets.
Since mixed-integer problems are in general NP-hard, numerically efficient approaches are required.
Thus, as the second contribution of this work, we introduce a new optimal covariance steering approach that is faster and requires fewer computations than the one proposed in \cite{okamoto2018Optimal}, which used the state values of all previous time steps and did not utilize the Markov property of the system dynamics.
Together, the two contributions allow us to solve path-planning problems in complex domains.
Our numerical simulations show how incorporating the covariance to the problem formulation changes the resulting optimal paths.


\section{Problem Statement\label{sec:ProblemStatement}}
This section formulates the general, non-convex chance-constrained optimal covariance steering problem.

\subsection{Problem Formulation}
The system dynamics consists of the following (possibly time-varying) discrete-time stochastic linear system with additive uncertainty,
\begin{equation}  \label{eq:SystemDynamics}
	x_{k+1} = A_kx_k + B_ku_k + D_kw_k, 
\end{equation}
where $k = 0,1,\ldots,N-1$ is the time step, $x\in\Re^{n_x}$ is the state, $u\in\Re^{n_u}$ is the control input, and $w\in\Re^{n_w}$ is a zero-mean white Gaussian noise with unit covariance, that is, 
$\mathbb{E}\left[w_k\right] = 0$, $\mathbb{E}\left[w_{k_1}{w_{k_2}^{\top}}\right] = I_{n_w}$ if $k_1 = k_2$, $\mathbb{E}\left[w_{k_1}{w_{k_2}^{\top}}\right] = 0$ otherwise. 
We also assume that $\mathbb{E}\left[x_{k_1}w_{k_2}^{\top}\right]=0$, $0 \leq k_1 \leq k_2 \leq N$.
The initial state $x_0$ is a random vector that is drawn from the multi-variate normal distribution
\begin{equation}\label{eq:x0}
	x_0\sim\mathcal{N}(\mu_0,\Sigma_0),
\end{equation}
where $\mu_0 \in \Re^{n_x}$ is the initial state mean and $\Sigma_0 \in \Re^{n_x \times n_x}$ is the initial state covariance. 
We assume that $\Sigma_0 \succeq 0$. 
Our objective is to steer the trajectories of the system~(\ref{eq:SystemDynamics}) from this initial distribution to the terminal Gaussian distribution
\begin{equation}\label{eq:xN}
	x_N\sim\mathcal{N}(\mu_N,\Sigma_N),
\end{equation}
where $\mu_N \in \Re^{n_x}$ and $\Sigma_N \in \Re^{n_x \times n_x}$ with $\Sigma_N \succ 0$, at a given time $N$, while minimizing the cost function
\begin{equation}  \label{eq:originalObjFunc}
	J(x_{0:N-1},u_{0:N-1}) = \mathbb{E}\left[\sum_{k=0}^{N-1}x_k^\top Q_k x_k + u_k^\top R_k u_k\right],
\end{equation}
where $x_{0:N-1}$ is the state sequence $x_0, \ldots, x_{N-1}$, $u_{0:N-1}$ is the control sequence $u_0, \ldots, _{N-1}$, and the matrices $Q_k \succeq 0$ and $R_k \succ 0$ for all $k=0,1,\ldots,N-1$.
Note that (\ref{eq:originalObjFunc}) does not include a terminal cost owing to the terminal constraint~(\ref{eq:xN}).

The objective is to compute the optimal control input sequence $u_0, u_1, \ldots, u_{N-1}$, which ensures that the probability of  the state violation at any given time is below a pre-specified threshold, that is, 
\begin{equation}\label{eq:originalChanceConstraints}
	\texttt{Pr}(x_k \notin \chi) \leq P_{\rm{fail}},\qquad k=0,\ldots,N-1, 
\end{equation}
where $\texttt{Pr}(\cdot)$ denotes the probability of an event, $\chi \subset \Re^{n_x}$ is the state constraint set, and $P_{\rm{fail}} \in [0, 1]$ is the threshold for the probability of failure. 
Optimization problems with these types of constraints are known as chance-constrained optimization problems~\cite{mesbah2016stochastic}.
Note that, unlike the problem setup in \cite{okamoto2018Optimal}, the set $\chi$ in the current work may be \emph{non-convex} owing to the presence of several obstacles, i.e.,
\begin{equation}\label{eq:non_convexX}
	\chi = \chi_\Omega \backslash \left( \cup_{j = 1}^{N_\mathrm{obs}}\chi_j \right),
\end{equation}
where $\chi_\Omega, \chi_1, \ldots, \chi_{N_{\mathrm{obs}}} \subset \Re^{n_x}$ are convex polytopes and $N_\mathrm{obs}$ is the number of obstacles.

\begin{Remark}
	System (\ref{eq:SystemDynamics}) is assumed to be controllable in the absence of (\ref{eq:non_convexX}), that is, $x_N$ is reachable from $x_0$ for any $x_N \in \Re^{n_x}$, provided that $w_k = 0$ for $k = 0,\ldots,N-1$. This assumption implies that given any $x_N \in \Re^{n_x}$ and $x_0 \in \Re^{n_x}$, there exists a sequence of bounded control inputs $\{u_k\}_{k=0}^{N-1}$ that steers $x_0$ to $x_N$ in the absence of disturbances and obstacles.
\end{Remark}

\subsection{Preliminaries \label{subsec:Preliminaries}}

We provide an alternative description of the system dynamics in (\ref{eq:SystemDynamics}), which will be instrumental for solving the problem.
At each time step $k$, we explicitly compute the system state $x_k$ as follows. 
Let $A_{k_1,k_0}$, $B_{k_1,k_0}$, and $D_{k_1,k_0}$, where $k_1>k_0$, denote the transition matrices of the state, input, and the noise term from step $k_0$ to step $k_1+1$, respectively, as
$A_{k_1,k_0} = A_{k_{1}}A_{k_{1}-1}\cdots A_{k_{0}}$, $B_{k_1,k_0} = A_{k_{1},k_{0}+1}B_{k_0}$, $D_{k_1,k_0} = A_{k_{1},k_{0}+1}D_{k_0}$.
We define the concatenated vectors $U_{k} = [u_{0}^\top \ u_{1}^\top \ \ldots \ u_{k}^\top]^\top \in \Re^{(k+1)n_u}$ and $	W_{k} = [w_{0}^\top \ w_{1}^\top\ \ldots \ w_{k}^\top]^\top \in \Re^{(k+1)n_w}$
Then $x_k$ can be equivalently computed from $x_k = \bar{A}_k x_0 + \bar{B}_k U_k + \bar{D}_k W_k$, where $\bar{A}_k = A_{k-1,0}$, $\bar{B}_k = [B_{k-1,0}\ B_{k-1,1}\ \cdots\ B_{k-1}]$, and $\bar{D}_k = [D_{k-1,0}\ D_{k-1,1}\ \cdots \ D_{k-1}]$.
Furthermore, we introduce the augmented state vector $X_k = [x_{0}^\top\ x_{1}^\top\ \ldots\ x_{k}^\top ]^\top \in \Re^{(k+1) n_x}$.
It follows that the system dynamics~(\ref{eq:SystemDynamics}) take the equivalent form
\begin{equation}\label{eq:convertedStateDynamics}
X = \ScriptA x_{0} + \ScriptB U+\ScriptD W, 
\end{equation}
where $X = X_{N}\in \Re^{(N+1) n_x} $, $U = U_{N-1}\in\Re^{Nn_u}$, and $W = W_{N-1}\in \Re^{Nn_w}$,
and the matrices $\ScriptA\in\Re^{(N+1)n_x\times n_x}$, $\ScriptB \in\Re^{(N+1)n_x\times Nn_u}$, and $\ScriptD \in\Re^{(N+1)n_x\times Nn_w}$ are defined accordingly.
Note that 
$\mathbb{E}[x_0x_0^\top] = \Sigma_0 + \mu_0\mu_0^\top$,
$\mathbb{E}[x_0W^\top] = 0$, and 
$\mathbb{E}[WW^\top] = I_{Nn_w}$.
Using $X$ and $U$, we may rewrite the cost function~(\ref{eq:originalObjFunc}) as 
\begin{equation}\label{eq:convertedObjFunc}
	J(X,U) = \mathbb{E}\left[X^\top \bar{Q} X + U^\top \bar{R} U \right],
\end{equation}
where $\bar{Q} = \texttt{blkdiag}(Q_0,\ldots,Q_{N-1},0)$ and $\bar{R} = \texttt{blkdiag}(R_0,\ldots,R_{N-1})$.
Since $Q_k\succeq0$ and $R_k\succ0$ for all $k = 0,\ldots,N-1 $, it follows that $\bar{Q}\succeq0$ and $\bar{R}\succ0$.
The boundary conditions~(\ref{eq:x0}) and (\ref{eq:xN}) also take the form 
\begin{subequations}
	\begin{align}
	\mu_0 &= E_0\mathbb{E}[X],\label{eq:mu0}\\
	\Sigma_0 &= E_0\left(\mathbb{E}[XX^\top] - \mathbb{E}[X]\mathbb{E}[X]^\top\right) E_0^\top,\label{eq:Sigma0}
	\end{align}\label{eq:X0}
\end{subequations}
and
\begin{subequations}
	\begin{align}
	\mu_N &= E_N\mathbb{E}[X],\label{eq:muN}\\
	\Sigma_N &= E_N\left(\mathbb{E}[XX^\top] - \mathbb{E}[X]\mathbb{E}[X]^\top\right) E_N^\top,\label{eq:SigmaN}
	\end{align}\label{eq:XN}
\end{subequations}
where $E_k \triangleq \left[0_{n_x,kn_x}, I_{n_x},0_{n_x,(N-k)n_x}\right]$. 
Finally, the chance constraints~(\ref{eq:originalChanceConstraints}) can be rewritten as
\begin{equation}  \label{eq:convertedChanceConstraints}
	\texttt{Pr}(E_k X \notin \chi) \leq P_{\rm{fail}}, \quad k = 0,\ldots,N-1.
\end{equation}

In summary, we solve the following problem.
\begin{problem}\label{prob:OriginalProblem}
	Given the system (\ref{eq:convertedStateDynamics}), find the control sequence $U^\ast = U_{N-1}^\ast$ that minimizes 
	the cost function Eq.~(\ref{eq:convertedObjFunc})  subject to the initial state constraints~(\ref{eq:X0}), the terminal state constraints~(\ref{eq:XN}), 
	and the chance constraint~(\ref{eq:convertedChanceConstraints}).
\end{problem}

\section{Optimal Covariance Control with Obstacles\label{sec:ProposedApproach}}
This section introduces the proposed approach to solve Problem \ref{prob:OriginalProblem}.
First, we introduce a computationally more efficient approach than the one in~\cite{okamoto2018Optimal}.
We then introduce how to deal with non-convex state chance constraints.

\subsection{New Covariance Steering Approach}\label{subsec:NewApproachVK}
In \cite{okamoto2018Optimal} we used the state values of all previous time steps, which did not utilize the Markov property of the system dynamics.
Thus, if the time horizon is long, large memory is required.
In this section, we propose a computationally more efficient approach.
The main result is given in the following theorem.
\begin{Theorem}\label{theorem:NewCovSteer}
	The following control sequence $U = [u_0^\top, u_1^\top,\ldots, u_{N-1}^\top]^\top$, where
	\begin{equation}
	u_k = v_k + K_k y_k,\label{eq:u_k=V + K_ky_k}
	\end{equation}
	where $v_k \in \Re^{n_u}$, $K_k \in \Re^{n_u \times n_x}$, and $y_k \in \Re^{n_x}$ is given by  
	\begin{subequations}
		\begin{align}
		y_{k+1} &= A_k y_k + D_k w_k,\label{eq:y_DynamicsExplicit}\\
		y_0 &= x_0 - \mu_0,\label{eq:y_IntExplicit}
		\end{align}
	\end{subequations}
	results in a convex programming formulation of Problem \ref{prob:OriginalProblem}.
\end{Theorem}

\begin{proof}
	The control sequence $U$ can be represented as $U = V + K Y$, where $V = [v_0^\top,\ldots, v_{N-1}^\top]^\top \in \Re^{Nn_u}$, $K\in\Re^{Nn_u \times (N+1) n_x}$, and $Y= [ y_0^\top,\ldots, y_{N}^\top ]^\top \in \Re^{(N+1)n_x}$.
	The new state $y_k$ is governed by the dynamics~(\ref{eq:y_DynamicsExplicit}) with initial condition~(\ref{eq:y_IntExplicit}). 
	Thus, using the matrices introduced in Section \ref{subsec:Preliminaries}, $Y$ can be written as $Y = \ScriptA y_0 + \ScriptD W$.
	Therefore, 
	\begin{equation}\label{eq:U=V+KAyDW}
		U = V + K \left(\ScriptA y_0 + \ScriptD W\right). 
	\end{equation}
	It follows from $\Expectation[y_0] = 0$ and $\Expectation[W] = 0$ that $\Expectation[U] = V.$
	Thus, it follows from (\ref{eq:convertedStateDynamics}) that
	\begin{align}
		\bar{X} &= \Expectation[X] = \ScriptA \mu_0 + \ScriptB V,\label{eq:barXVK}\\
		\tilde{X} &= X -\Expectation[X] = \ScriptA (x_0 - \mu_0) + \ScriptB (U-V) + \ScriptD W,\nonumber \\
		          &= (I + \ScriptB K )\left(\ScriptA y_0 + \ScriptD W\right).\label{eq:tildeXVK}
	\end{align}
	According to~\cite{okamoto2018Optimal}, the cost function~(\ref{eq:convertedObjFunc}) may be converted into $J(\bar{X},\tilde{X},V,\tilde{U}) = \trace(\bar{Q}\Expectation[\tilde{X}\tilde{X}^\top]) + \bar{X}^\top \bar{Q} \bar{X}
		+ \trace(\bar{R}\Expectation[\tilde{U}\tilde{U}^\top]) + V^\top \bar{R} V$,
	where $\tilde{U} = U - V$.
	It follows from $\Expectation[y_0y_0^\top] = \Sigma_0$, $\Expectation[y_0W^\top] = 0$, and $\Expectation[WW^\top] = I_{Nn_w}$ that 
		$\Expectation[\tilde{X}\tilde{X}^\top] = (I + \ScriptB K )\left(\ScriptA \Sigma_0 \ScriptA^\top + \ScriptD\ScriptD^\top\right)(I + \ScriptB K )^\top$ and
		$\Expectation[\tilde{U}\tilde{U}^\top] = K\left(\ScriptA\Sigma_0\ScriptA^\top + \ScriptD\ScriptD^\top\right)K^\top$.
	Therefore, the cost function is further converted to the following quadratic form in terms of $V$ and $K$:
	\begin{align}
		&J(V,K) = \nonumber\\
		&\trace\left(\left((I+\ScriptB K)^\top\bar{Q}(I + \ScriptB K ) + K^\top\bar{R}K\right)\left(\ScriptA \Sigma_0 \ScriptA^\top + \right.\right. \nonumber\\
		 &\left.\left.\ScriptD\ScriptD^\top\right)\right) + (\ScriptA \mu_0 + \ScriptB V)^\top\bar{Q}(\ScriptA \mu_0 + \ScriptB V) + V^\top\bar{R}V.\label{eq:costFunctionVandK}
 	\end{align}
 
	In addition, the terminal state constraints~(\ref{eq:XN}) can be formulated as
	\begin{subequations}
		\begin{align}
		\mu_N &= E_N\left(\ScriptA \mu_0 + \ScriptB V\right),\label{eq:muNNew}\\
		\Sigma_N &= E_N (I + \ScriptB K )\left(\ScriptA \Sigma_0 \ScriptA^\top + \ScriptD\ScriptD^\top\right)(I + \ScriptB K )^\top E_N^\top.
		\end{align}\label{eq:XNNew}
	\end{subequations}
	Note that $V$ steers the mean and $K$ steers the covariance to the pre-specified values $\mu_N$ and $\Sigma_N$, respectively. 
	As we discussed earlier, in order to reformulate Problem~\ref{prob:OriginalProblem} as a convex programming problem, we relax the covariance equality constraint~(\ref{eq:XNNew}) to an inequality constraint. 
	Thus, $\Sigma_N \succeq E_N (I + \ScriptB K )\left(\ScriptA \Sigma_0 \ScriptA^\top + \ScriptD\ScriptD^\top\right)(I + \ScriptB K )^\top E_N^\top$, 
	which leads to
	\begin{equation} \label{eq:XNNewIneq}
	1 - \| (\mathcal{A}\Sigma_0\mathcal{A}^\top+\mathcal{D}\mathcal{D}^\top)^{1/2}(I+\mathcal{B}K)^\top E_N^\top\Sigma_N^{-1/2}  \|_2 \geq 0.
	\end{equation}
	
	Finally, the chance constraint~(\ref{eq:convertedChanceConstraints}) can be formulated as follows. 
	Assuming $\chi$ is defined as the intersection of $M$ linear inequality constraints $\chi \triangleq \bigcap_{j=1}^M \{x:\alpha_j^\top x \leq \beta_j\}$
	where $\alpha_j \in \Re^{n_x}$ and $\beta_j \in \Re$, the chance constraint~(\ref{eq:convertedChanceConstraints}) is converted to $\texttt{Pr}(\alpha_j^\top E_k X > \beta_j) \leq p_j$ and $\sum_{j=1}^{M}p_j \leq P_{\rm{fail}}$.
	It follows from (\ref{eq:barXVK}) and (\ref{eq:tildeXVK}) that $\alpha_j^\top E_k X$ is a univariate Gaussian random variable such that $\alpha_j^\top E_kX \sim \mathcal{N}(\alpha_j^\top E_k \bar{X}, \alpha_j^\top E_k \Sigma_X E_k^\top \alpha_j)$, where 
	\begin{align}\label{eq:Sigma_X}
		\Sigma_X = (I+\ScriptB K)(\ScriptA\Sigma_0\ScriptA^\top + \ScriptD\ScriptD^\top)(I+\ScriptB K)^\top.
	\end{align}
	Thus, it is straightforward from the discussion in~\cite{okamoto2018Optimal} to derive the following inequality constraint 
	\begin{align}
		&\alpha_j^\top E_k\left(\ScriptA \mu_0 + \ScriptB V\right)+ \|(\mathcal{A}\Sigma_0\mathcal{A}^\top+\mathcal{D}\mathcal{D}^\top)^{1/2} \nonumber \\
		&(I+\mathcal{B}K)^\top E_k^\top\alpha_j\| \Phi^{-1}\left(1-p_{j,\rm{fail}}\right) - \beta_j \leq 0, \label{eq:deterministicChanceConstraintsOursNew}
	\end{align}
	where $p_{j, {\rm fail}}$ is a pre-specified value.
	
	In summary, by introducing (\ref{eq:u_k=V + K_ky_k}), we have converted Problem~\ref{prob:OriginalProblem} into the following convex programming problem.
\end{proof}	
\begin{problem}\label{prob:FasterProblem}
Given the system state sequence Eqs.~(\ref{eq:barXVK}) and (\ref{eq:tildeXVK}), find $V$ and $K$ that minimizes 
the cost function Eq.~(\ref{eq:costFunctionVandK})  subject to the terminal state constraints~(\ref{eq:muNNew}) and (\ref{eq:XNNewIneq}), 
and the chance constraint~(\ref{eq:deterministicChanceConstraintsOursNew}) with pre-specified probability thresholds $p_{j,\rm{fail}}$.
\end{problem}

\begin{Remark}
	The new control design strategy based on Theorem~\ref{theorem:NewCovSteer} is computationally more efficient than the one proposed in~\cite{okamoto2018Optimal}.
	Specifically, the vector $V$ and the matrix $K$ in (\ref{eq:U=V+KAyDW}) are a full vector and a block diagonal matrix.
	Thus, the number of entries one needs to compute is $\mathcal{O}(n_xn_uN)$.
	In contrast, the matrix $K$ in~\cite{okamoto2018Optimal} has $\mathcal{O}(N^2n_xn_u)$
	entries.
\end{Remark}
\begin{Remark}
	$y_k$ is the uncontrolled ($u_k \equiv 0$) state of (\ref{eq:SystemDynamics}) with a different initial condition~(\ref{eq:y_IntExplicit}).
\end{Remark}
\begin{Remark}
	It is possible from (\ref{eq:costFunctionVandK}) to separately design the mean and covariance steering cost matrices as
	\begin{align}
	&J(V,K) = \trace\left(\left((I+\ScriptB K)^\top\bar{Q}_\mathrm{cov}(I + \ScriptB K ) + K^\top\bar{R}_\mathrm{cov}K\right)\right.\nonumber \\
	&\left.\left(\ScriptA \Sigma_0 \ScriptA^\top + \ScriptD\ScriptD^\top\right)\right) + (\ScriptA \mu_0 + \ScriptB V)^\top\bar{Q}_\mathrm{mean}(\ScriptA \mu_0 + \ScriptB V) \nonumber \\ 
	&\hspace{20pt} + V^\top\bar{R}_\mathrm{mean}V.\label{eq:costFunctionVandKSeparate}
	\end{align}
\end{Remark}

\subsection{Optimal Covariance Control with Obstacles}
In this section, we propose a new approach to efficiently deal with non-convex state chance constraints.
As discussed in Eq.~(\ref{eq:non_convexX}), and unlike \cite{okamoto2018Optimal}, we assume that the state space has several convex obstacles. 
In Section~\ref{sec:Introduction} we discussed numerous approaches that have been proposed to deal with path planning problems with obstacles.
However, most of them only consider the deterministic system and do not steer the covariance, which may lead to unnecessarily conservative solutions as we demonstrate later in Section~\ref{sec:Numerical Simulation} via numerical examples.
By steering the state covariance along with the mean, we mitigate the conservativeness of the ensuing path.
We first introduce an approach to deal with the non-convex state chance constraints~(\ref{eq:convertedChanceConstraints}) and (\ref{eq:non_convexX}).
Then, we convert the original Problem~\ref{prob:OriginalProblem} into a mixed-integer convex programming problem.

We represent the non-convex set of obstacle-free states as the union of a finite number of (possibly overlapping) \emph{convex} regions $\mathcal{R}_r$.
Specifically, we represent $\chi$ in (\ref{eq:non_convexX}) as 
\begin{equation} \label{eq:NonConvexDecomposed}
	\chi = \bigcup_{r=0}^{N_R-1}\mathcal{R}_r,
\end{equation}
where $N_R$ is the number of convex regions. 
We assume that each convex region $\mathcal{R}_r$ is a polytope and can be represented as the intersection of multiple linear inequality constraints as follows.
\begin{equation}\label{eq:R_rDef}
	\mathcal{R}_r \triangleq \bigcap_{q = 0}^{M_r-1} \{x:\alpha_{r,q}^\top x \leq \beta_{r,q}\},
\end{equation}
where $\alpha_{r,q} \in \Re^{n_x}$ and $\beta_{r,q} \in \Re$. 
Let the  Boolean matrix $\mathcal{M}\in\{0,1\}^{N_R \times (N-1)}$, where $\mathcal{M}_{r,k} = 1$ implies that the state at time steps $k$ and $k+1$ belongs to region $\mathcal{R}_r$ with high probability. 
Note that, because of the noise, the state constraints need to be probabilistically formulated, i.e., using chance constraints.
Namely, we impose that 
\begin{equation} \label{eq:M->EkXinRr}
	\mathcal{M}_{r,k} = 1 \Rightarrow \mathtt{Pr}\left(x_k \notin \mathcal{R}_r\right) < \epsilon ~\mathrm{and}~\mathtt{Pr}\left(x_{k+1} \notin \mathcal{R}_r\right) < \epsilon,
\end{equation}
where $0 < \epsilon \ll 1$.
In order to ensure that, with high probability, the state is collision-free at time step $k$, we use the following equality constraint
\begin{equation}\label{eq:condOnM}
	\sum_{r=0}^{N_R-1}\mathcal{M}_{r,k} = 1.
\end{equation}
Note that, as there can be overlaps between regions, the state variables at steps $k$ and $k+1$ can belong to multiple regions.
Thus, the implication in (\ref{eq:M->EkXinRr}) is only one directional \cite{deits2015efficient}.


Next, we prove that the right-hand side of (\ref{eq:M->EkXinRr}) can be formulated as a constraint in a mixed-integer convex programming problem.
\begin{Lemma}
	Given $\mathcal{R}_r$ in (\ref{eq:R_rDef}), the condition
	$\mathtt{Pr}\left(x_k \notin \mathcal{R}_r\right) < \epsilon$ and $\mathtt{Pr}\left(x_{k+1} \notin \mathcal{R}_r\right) < \epsilon$,
	where the state $x_k =E_k (\bar{X} + \tilde{X})= E_k(\ScriptA \mu_0 + \ScriptB V +  (I + \ScriptB K)(\ScriptA y_0 + \ScriptD W))$, 
	is a second-order cone constraint in $V$ and $K$:
\end{Lemma}
\begin{proof}
	It is easy to show, from the discussion in Section~\ref{subsec:NewApproachVK}, that the condition can be converted to 
	\begin{equation}
		\begin{cases}
			\alpha_{r,q}^\top E_k\left(\ScriptA \mu_0 + \ScriptB V\right)+ \|(\mathcal{A}\Sigma_0\mathcal{A}^\top+\mathcal{D}\mathcal{D}^\top)^{1/2}\\
			\hspace{5pt} (I+\mathcal{B}K)^\top E_k^\top\alpha_{r,q}\| \Phi^{-1}\left(1-\epsilon\right) - \beta_{r,q} \leq 0,\\
			\alpha_{r,q}^\top E_{k+1}\left(\ScriptA \mu_0 + \ScriptB V\right)+ \|(\mathcal{A}\Sigma_0\mathcal{A}^\top+\mathcal{D}\mathcal{D}^\top)^{1/2} \\
			\hspace{5pt}(I+\mathcal{B}K)^\top E_{k+1}^\top\alpha_{r,q}\| \Phi^{-1}\left(1-\epsilon \right) - \beta_{r,q} \leq 0,
		\end{cases}\label{eq:HChanceConstraints}
	\end{equation}
	where $q = 0,\ldots,M_r-1$. This condition is second-order cone in $V$ and $K$.
\end{proof}

Finally, we reformulate Problem \ref{prob:OriginalProblem} into the following mixed-integer convex programming problem summarized in Algorithm~\ref{alg:ProposedApproach}.
Note that, although MIP problems are, in general, NP-hard, many tools have been recently developed in the literature to efficiently solve such problems. 
In the following section, using simple numerical examples, we demonstrate that our problem setup can be efficiently solved with current MIP solvers.
Specifically, we use YALMIP~\cite{yalmip}, which uses a branch-and-bound algorithm to solve MIP problems, and bounding the design variables is helpful to find a solution. 
Thus, we introduce element-wise constraints on $V$ and $K$ (line 5).


\begin{algorithm}
\DontPrintSemicolon
\SetKwInOut{Input}{Input}
\SetKwInOut{Output}{Output}
\SetKwComment{Comment}{$\triangleright$\ }{}
\Input{$\mu_0$, $\Sigma_0$, $\mu_N$, $\Sigma_N$, $\epsilon$, $A_k,B_k,D_k$, $\alpha_{r,q}$, $\beta_{r,q}$. ($k=0,\ldots,N-1$, $r=0,\ldots,N_R-1$, $q = 0,\ldots,M_r-1$)}
\Output{$V^\ast$ and $K^\ast$}
	Find $V$ and $K$ in~(\ref{eq:U=V+KAyDW}) that minimizes (\ref{eq:costFunctionVandK}) or (\ref{eq:costFunctionVandKSeparate}) \\
	subject to \\
	\hspace{10pt}(\ref{eq:barXVK}) and (\ref{eq:Sigma_X}) \Comment{System Dynamics}
	\hspace{10pt}(\ref{eq:muNNew}) and (\ref{eq:XNNewIneq}) \Comment{Terminal Condition}
	\hspace{10pt}$V_\mathrm{min} \leq V \leq V_\mathrm{max}$, $K_\mathrm{min} \leq K \leq K_\mathrm{max}$\\
	\hspace{10pt}(\ref{eq:condOnM}) \Comment{Condition on $\mathcal{M}$}
	\hspace{10pt}$\mathcal{M}_{r,k} == 1$ $\Rightarrow$(\ref{eq:HChanceConstraints}) \Comment{Chance Constraints}
$K^\ast$ $\gets$ $K$. $V^\ast$ $\gets$ $V$.
\caption{Optimal Covariance Control}\label{alg:ProposedApproach}
\end{algorithm}

\section{Numerical Simulations\label{sec:Numerical Simulation}}
In this section we validate the proposed algorithm using simple numerical examples. 
We consider the path-planning problem for a vehicle under the following time invariant system dynamics with $x_k = [x,y,v_x,v_y]^\top\in \Re^{4}$, $u_k =[a_x,a_y]^\top\in \Re^{2}$, $w_k \in \Re^{4}$ and
\begin{equation}
	A = \begin{bmatrix}
	1 & 0 & \Delta t & 0\\
	0 & 1 & 0 & \Delta t\\
	0 & 0 & 1 & 0\\
	0 & 0 & 0 & 1
	\end{bmatrix},
	B = \begin{bmatrix}
	\Delta t^2/2 & 0\\
	0 &\Delta t^2/2\\
	\Delta t & 0\\
	0 & \Delta t
	\end{bmatrix},\label{eq:VehicleDynamics}
\end{equation}
and $D = 10^{-2}I_4$, where $\Delta t = 0.2$ is the time-step size. 
We use the cost function in (\ref{eq:costFunctionVandKSeparate}), and the cost function weights are chosen as 
$\bar{Q}_\mathrm{mean} = \mathtt{blkdiag}(Q_0,Q_1,\ldots,Q_{N-1},0)$ with $Q_k = \mathtt{diag}(0.5,4.0,0.05,0.05)$, $\bar{Q}_\mathrm{cov} = 0$,
$\bar{R}_\mathrm{mean} = \mathtt{blkdiag}(R_0,R_1,\ldots,R_{N-1})$ with $R_k = \mathtt{diag}(20,20)$, and 
$\bar{R}_\mathrm{cov} = \mathtt{blkdiag}(R_{\mathrm{cov},0},R_{\mathrm{cov},1},\ldots,R_{\mathrm{cov},N-1})$ with $R_{\mathrm{cov},k} = \mathtt{diag}(200,200)$ for $k=0,\ldots,N-1$.
We also set the horizon to $N=20$ and the probability threshold to $\epsilon = 1e-3$.
We used YALMIP~\cite{yalmip} with MOSEK~\cite{mosek} to solve the relevant optimization problems.
In order to reduce the search space and the computational time, we restrict the following element-wise inequality constraint to the control vector and gain matrix
$-10 \leq K \leq 10$, $-100 \leq V \leq 100$.
\subsection{Double Slit}\label{subsec:Scenario2}
We consider the case illustrated in Fig.~\ref{fig:2ProblemSetupWide}, where we find the trajectory to go through a ``slit.''
The red circle denotes the $3\sigma$ error of the initial state distribution of the $x$ and $y$ coordinates.
The magenta circle denotes the $3\sigma$ error of the terminal state distribution of $x$ and $y$ coordinates.
The initial condition is $\mu_0 = [-10, 0.1, 0, 0]$ and $\Sigma_0 = \mathtt{diag}(0.05,0.05,0.001,0.001)$,
while the terminal constraint is $\mu_N = [0, 0, 0, 0]$ and $\Sigma_N = \mathtt{diag}(0.01,0.01,0.001,0.001)$.
First, for comparison, we conduct only mean steering as shown in Fig.~\ref{fig:Problem2WideOL}.
This result is obtained by imposing $K\equiv0$ in (\ref{eq:U=V+KAyDW}). 
In this case, the covariance is not controlled, and thus, the terminal covariance constraint cannot be satisfied.
Therefore, the terminal covariance constraint (\ref{eq:XNNewIneq}) is not imposed.
Since the initial covariance is large, and as the covariance grows, it is impossible for the mean steering controller to guide the trajectory through the top slit, and the path has to go through the larger but further away slit.
Figure~\ref{fig:Problem2WideCL} illustrates the result of the proposed approach.
Although the top slit is narrower than the bottom one, the controller shrinks the covariance and successfully computes an optimal path. 
In order to compute this path, we used the rectangular-shaped sub-regions shown in Fig.~\ref{fig:Ex2WideRegion}.
Although this division was conducted manually, algorithms are available to automatically represent the entire feasible region as a union of convex polytopes \cite{deits2015computing}.

We also conducted a similar numerical simulation with slightly different setup, where the top slit is much narrower.
As illustrated in Fig.~\ref{fig:2ProblemSetupNarrow}, the algorithm found that the cost of shrinking the error ellipses and going through the top slit is higher than going through the lower slit as shown in Fig.~\ref{fig:2SolutionNarrow}.
\begin{figure}
	\centering
	\subfloat[Problem setup.\label{fig:2ProblemSetupWide}]{\centering\includegraphics[width=0.5\columnwidth]{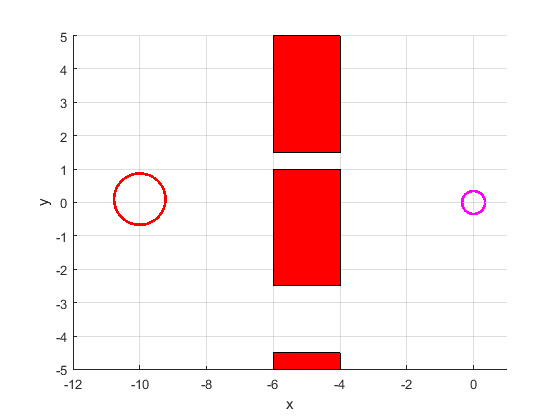}}\\
	\subfloat[Mean steering.\label{fig:Problem2WideOL}]{\centering\includegraphics[width=0.5\columnwidth]{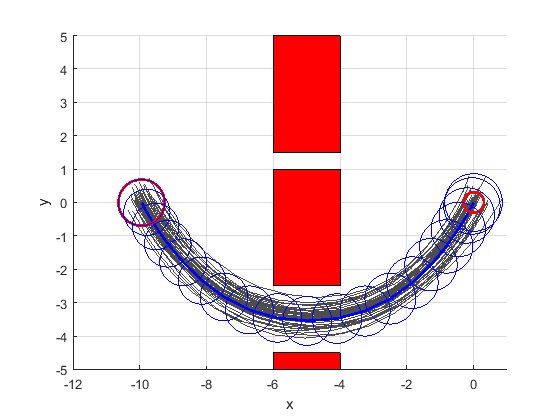}}
	\subfloat[Proposed approach.\label{fig:Problem2WideCL}]{\centering\includegraphics[width=0.5\columnwidth]{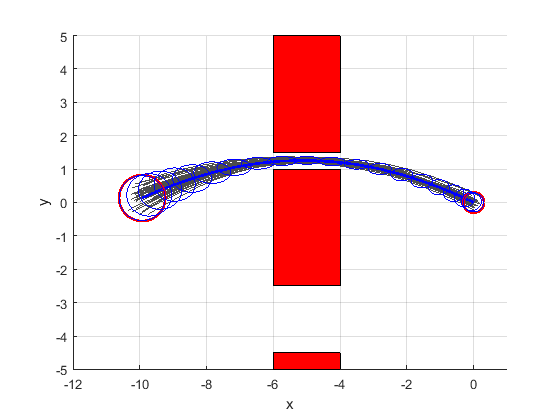}}
	\caption{Problem setup and solutions.}
\end{figure}

\begin{figure}
	\centering
	\subfloat[Region 1]{\centering\includegraphics[width=0.5\columnwidth]{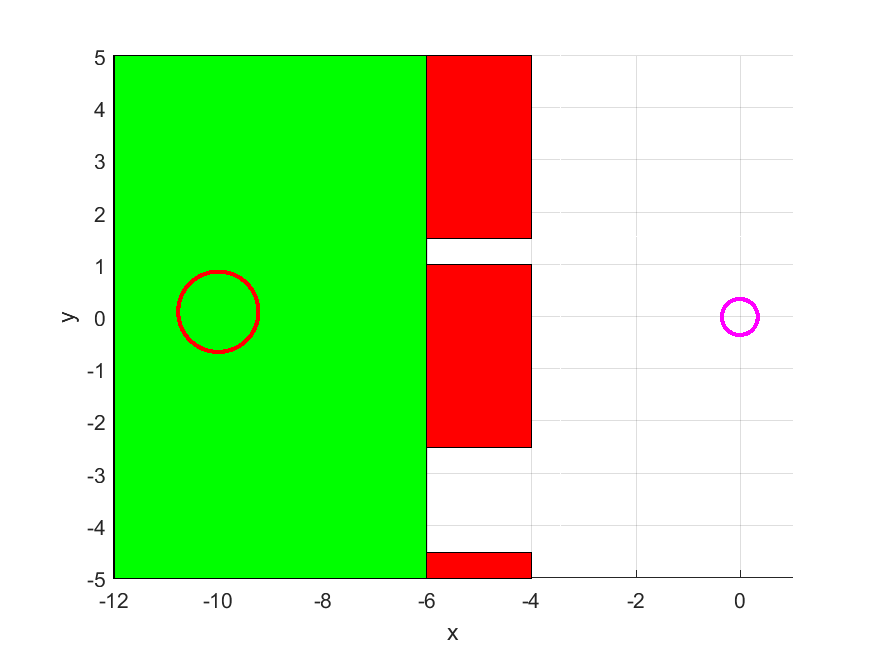}}
	\subfloat[Region 2]{\centering\includegraphics[width=0.5\columnwidth]{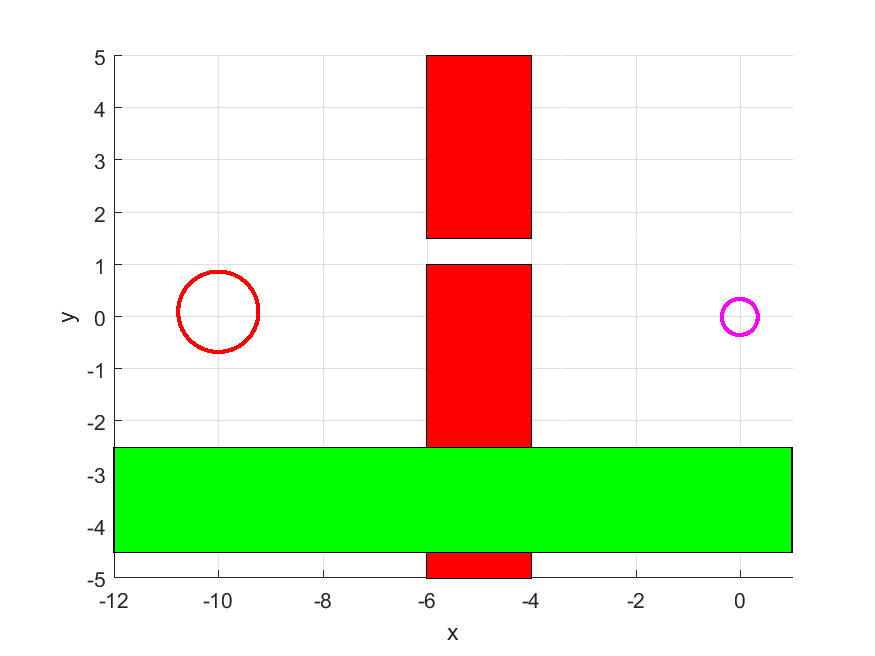}}\\
	\subfloat[Region 3]{\centering\includegraphics[width=0.5\columnwidth]{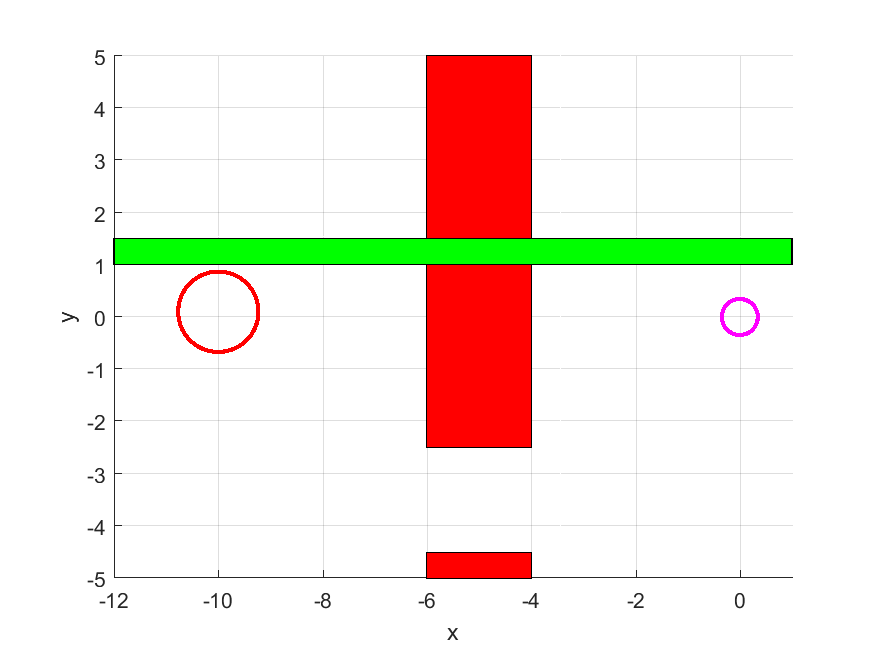}}
	\subfloat[Region 4]{\centering\includegraphics[width=0.5\columnwidth]{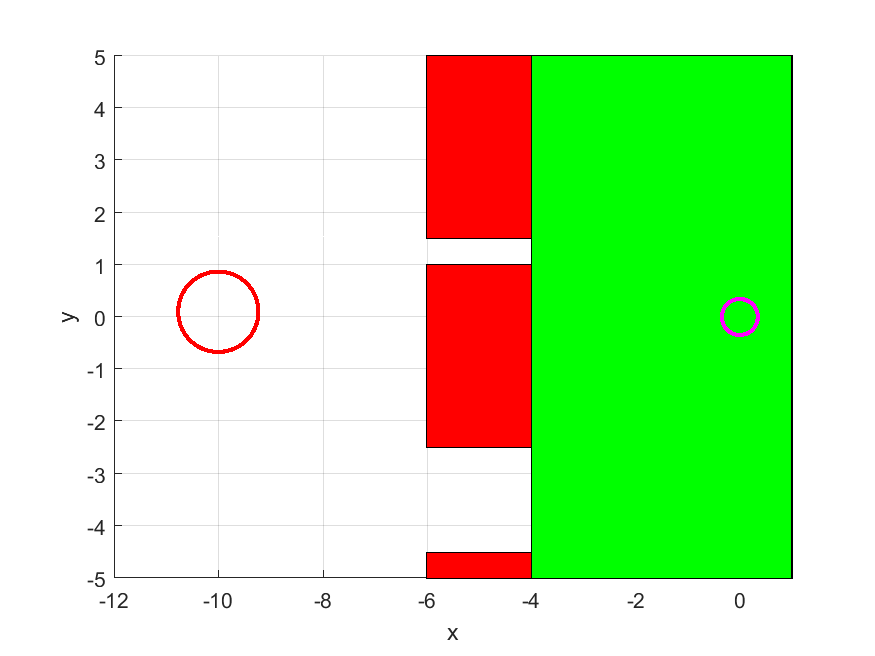}}
	\caption{Representation of the feasible region.\label{fig:Ex2WideRegion}}
\end{figure}

\begin{figure}
	\centering
	\subfloat[Problem setup.\label{fig:2ProblemSetupNarrow}]{\centering\includegraphics[width=0.5\columnwidth]{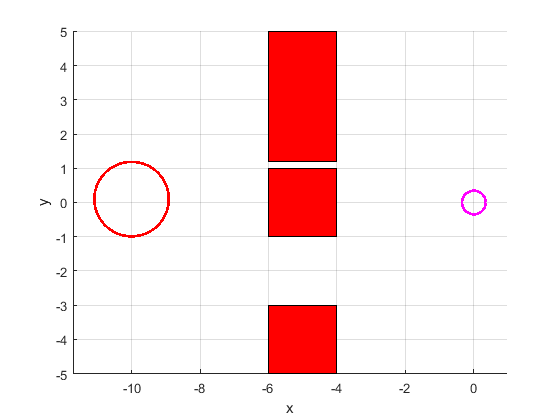}}
	\subfloat[Solution of the proposed approach.\label{fig:2SolutionNarrow}]{\centering\includegraphics[width=0.5\columnwidth]{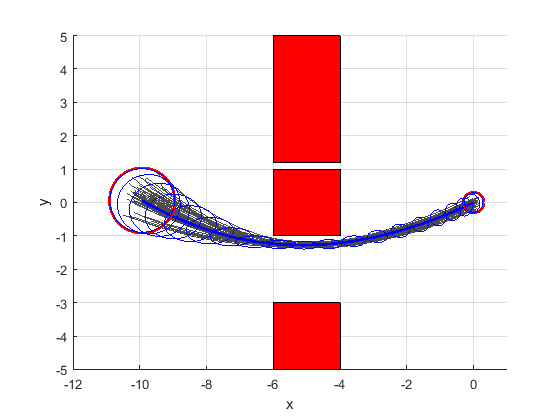}}
	\caption{Problem setup and solution.}
\end{figure}

\subsection{Cluttered Environment}
We consider the case illustrated in Fig.~\ref{fig:ProblemCluttered}, where we computed the optimal collision-free trajectory in a somewhat cluttered environment.
The initial condition is $\mu_0 = [-10, 0.1, 0, 0]$, $\Sigma_0 = \mathtt{diag}(0.05,0.05,0.001,0.001)$,
while the terminal constraint is $\mu_N = [0, 0, 0, 0]$, $\Sigma_N = \mathtt{diag}(0.01,0.01,0.001,0.001)$.
First, for comparison, we illustrate the result of a mean steering controller in Fig.~\ref{fig:ProblemClutteredOL}.
This result is obtained by imposing $K\equiv0$ in (\ref{eq:U=V+KAyDW}). 
In this case, the covariance is not controlled, and thus, the terminal covariance constraint cannot be satisfied.
Therefore, the terminal covariance constraint (\ref{eq:XNNewIneq}) is not imposed.
The controller finds that the ``corridor'' between $y=0$ and $y=1$ is too narrow and the path needs to detour to the top region.
In contrast, the proposed approach shrinks the covariance and successfully follows the shortest corridor while satisfying the non-convex state constraints as shown in Fig.~\ref{fig:ProblemClutteredCL}

\begin{figure}
	\centering
	\subfloat[Problem Setup.\label{fig:ProblemCluttered}]{\centering\includegraphics[width=0.5\columnwidth]{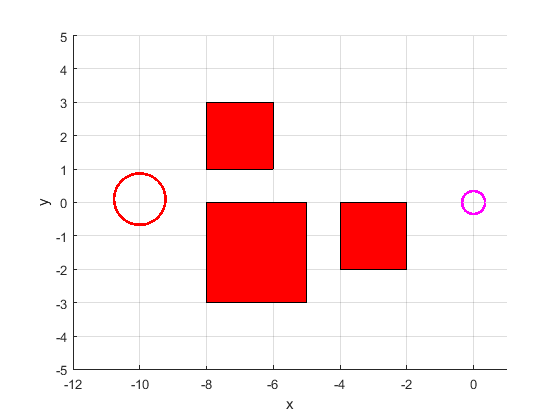}}\\
	\subfloat[Mean steering.\label{fig:ProblemClutteredOL}]{\centering\includegraphics[width=0.5\columnwidth]{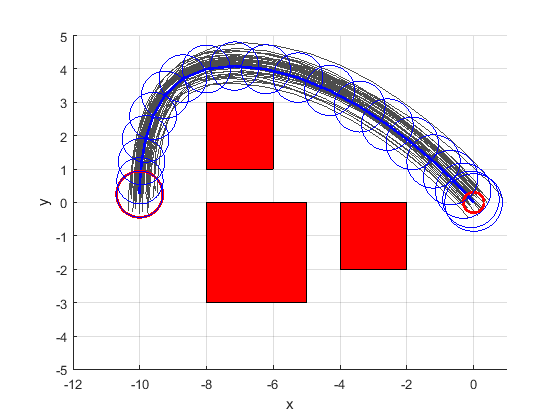}}
	\subfloat[Proposed approach.\label{fig:ProblemClutteredCL}]{\centering\includegraphics[width=0.5\columnwidth]{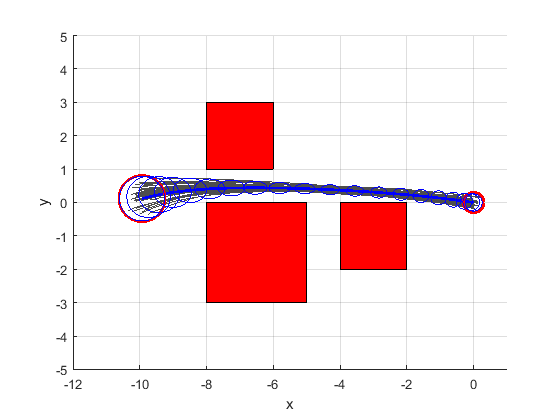}}
	\caption{Problem setup and solution.}
\end{figure}

\section{Discussion and Conclusion}\label{sec:Discussion&Conclusion}
The proposed approach is a non-trivial extension of our prior work \cite{okamoto2018Optimal}, where we addressed the problem of optimal covariance steering under \emph{convex} state chance constraints. 
The approach in~\cite{okamoto2018Optimal} could not deal with \emph{non-convex} state chance constraints.
Thus, in order to apply covariance steering to path planning problems, a different approach was needed.
Deterministic path-planning algorithms typically use Boolean variables to indicate collision with obstacles. 
As a result, the problem is converted to a mixed-integer programming problem \cite{richards2002spacecraft}.
It is known, however, that this approach typically needs separate integer variables for each face of each obstacle, which leads to a large computational overhead.
Here, we employed, instead, an approach similar to \cite{deits2015efficient}, in which integer variables indicate in which sub-region the state variable exists at each time step, leading to a much lower computational cost.
In the process, we also proposed a computationally more efficient calculation of the gain matrix than the one proposed in \cite{okamoto2018Optimal}.

We need to stress that, because we represent the feasible state space as the union of feasible convex sub-regions~(\ref{eq:NonConvexDecomposed}), the solution may become conservative, as illustrated in Fig.~\ref{fig:2ProblemNarrowZoom}, which shows the solution in Fig.~\ref{fig:2ProblemSetupNarrow} around the lower left corner of the middle obstacle.
The green areas are the feasible regions, and the red area is the obstacle.
The black lines indicate the edges of the sub-regions.
The error ellipse touches these boundaries, but it does not touch the corner of the obstacle, which indicates that the path has unnecessarily large safety margin from the obstacle.
This extra margin is owing to the requirement that this error ellipse needs to belong to both Regions 1 and 2 (see Fig.~\ref{fig:Ex2WideRegion}).
The decomposition in~(\ref{eq:NonConvexDecomposed}) is not unique. 
An interesting question would be to find the ``best'' decomposition to a union of convex sets for our problem.

\begin{figure}
	\vspace{-10pt}
	\centering
	\includegraphics[width=0.5\columnwidth]{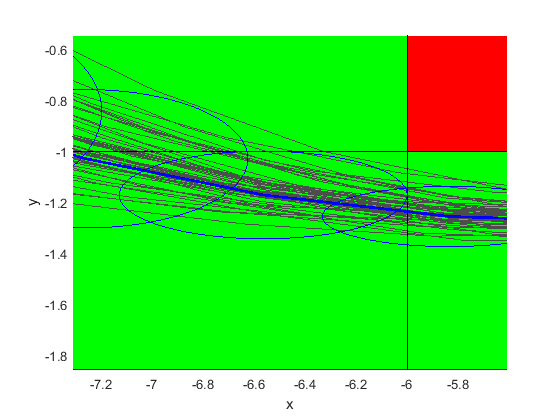}
	\caption{Zoom up of the solution illustrated in Fig.~\ref{fig:2SolutionNarrow}.\label{fig:2ProblemNarrowZoom}}
	\vspace{-10pt}
\end{figure}

In summary, in this work, we have addressed the problem of optimal covariance steering under non-convex state chance constraints.
We proposed to solve this problem by converting the original problem into a mixed-integer convex programing problem, which can be efficiently solved using an optimization solver.
In our numerical simulations, the proposed algorithm successfully found collision-free paths.
To the best of our knowledge, this is the first work to solve the optimal covariance steering problem with non-convex state chance constraints.
Future work includes the investigation of an effective approach to separate the feasible state space to a union of convex sets.

\vspace*{2ex}
\textbf{Acknowledgment:}
This work has been supported by ARO award W911NF-16-1-0390 and NSF award CPS-1544814.
K. Okamoto was also supported by Funai Foundation for Information Technology and Ito Foundation USA-FUTI Scholarship.
\bibliographystyle{IEEEtran}
\bibliography{NonConvexCCCS}

\end{document}